\documentclass{article}
\usepackage[utf8]{inputenc} 
\usepackage{amsmath,amssymb} 
\usepackage{amsthm}
\newtheorem*{note}{Note}
\newtheorem{theorem}{Theorem}
\usepackage{natbib}
\title{Necessary and Sufficient Conditions for Proving Choice in Zermelo-Fraenkel Set Theory}
\author{Valentyn Khokhlov}
\date{15 February 2026}

\begin{document}

\maketitle

\begin{abstract}
This paper introduces an alternative approach to proving the existence of choice functions for specific families of sets within Zermelo-Fraenkel set theory (ZF) without assuming any form on the Axiom of Choice (AC). Traditional methods of proving choice, when it is possible without AC, are based on explicit constructing a choice function, which relies on being able to identify canonical elements within the sets. 
Our approach, instead, employs the axiom schema of separation. We begin by considering families of well-ordered sets, then apply the schema of separation twice to build a set of possible candidates for the choice functions, and, finally, prove that this set is non-empty. This strategy enables proving the existence of choice function in situations where canonical elements cannot be identified explicitly. 
We then extend our method beyond families of well-ordered sets to families of sets, over which partial orders with a least element exist. After exploring possibilities for further generalization, we establish a necessary and sufficient condition: in ZF, without assuming AC, a choice function exists for a non-empty family if and only if each set admits a partial order with a least element. Finally, we demonstrate how this approach can be used to prove the existence of choice functions for families of contractible and path-connected topological spaces, including hyper-intervals in $\mathbb{R}^n$, hyper-balls, and hyper-spheres.

\end{abstract}

\section{Introduction}
The Zermelo-Fraenkel set theory (ZF) is widely recognized as the most common foundation of mathematics. Its axiomatic framework was first established in \cite{Zermelo1908} and subsequently refined in numerous works, with one of the most comprehensive formulation presented in \cite{Fraenkel1973}. Nowadays it is common to add the axiom of choice (AC) to this theory in order to build a complete mathematical framework (ZFC), however  historically AC was perceived as controversial due to the paradoxes it can generate. As demonstrated in \cite{Cohen1964}, AC is independent of the other ZF axioms. The idea of AC predates ZF itself: Zermelo first suggested the idea of choice in \cite{Zermelo1904} while attempting to prove that every set can be well-ordered. Some stronger or weaker forms of AC has been proposed, such as axiom of global choice, axiom of dependable choice, and axiom of countable choice. Still, it is generally accepted that some kind of this axiom is necessary for building a solid foundation for mathematics.

The axiom of choice can be formulated in several equivalent ways, and in this paper we will use the following expression \eqref{AC}, which literally states ``\textit{for every non-empty family of non-empty sets there exists a function that maps each set of that family into an element of that set}":
\begin{equation}\label{AC}
\begin{split}
\forall{A_S} \Big[&\big(A_S \ne \emptyset \land \forall{A \in A_S} ( A \ne \emptyset)\big) \implies \\& \exists {f_c \in \mathcal{P}(A_S \times \bigcup_{A \in A_S}{A} ) } \big[ \forall{A \in A_S} \exists! {a \in A} \big((A,a) \in f_c\big) \big]  \Big]
\end{split}
\end{equation}

As noted in \cite{Hrbacek1999}, building on the original works of Zermelo, it can be shown within ZF alone---without assuming AC---that the power set of any well-ordered set admits a choice function. Moreover, the authors also demonstrate that every finite familiy $A_S$ possess a choice function without invoking AC. However, both of those proofs are based on constructing a choice function in a specific way. In fact, they essentially depend on being able to construct a function that picks the $\prec$-least element for a particular $\prec$ well-ordering on the given set. We refer to such proofs as ``constructive", since they explicitly construct a specific function based on  properties derivable from ZF alone.

Using the ``constructive" approach, it can be shown a choice function exists (i.e. can be constructed) if each set in the family has some natural ordering or a canonical element. For example, statement \eqref{AC} can be proved as a theorem in ZF if $A_S$ in finite, or even if it contains countably many subsets of $\mathbb{N}$, since the natural well-ordering on $\mathbb{N}$ can be used to uniquely determine the least element of each set. That suggests we might be able to prove choice for any families of countable sets without assuming AC. However, when $A_S$ contains sets with no natural well-ordering, even though some well-orders may exist, selecting a particular one among them will require some form of choice. Under AC, any set can be well-ordered by choosing its elements sequentially. However, the opposite is not automatically true: if we have got a family of well-ordered sets, not all of which are countable, that does not imply we can assume the existence of a choice function without AC when following a ``constructive" approach.

Therefore, the following question remains open, at least within the publicly available body of knowledge: if $A_S$ is a family of well-ordered sets, not all of which are countable (or more generally, when we cannot develop a ``constructive" proof for the reasons discussed above), can statement \eqref{AC} be proved from ZF alone, without assuming AC? This paper proposes an alternative to the ``constructive" approach, aiming to prove \eqref{AC} as a theorem in ZF under certain conditions. We develop proofs for specific cases involving well-ordered sets and partially ordered sets with a least element, and discusses possible extensions to some classes of unordered sets. The paper is structured as follows: in Section 2 we specify the axiomatic basis for our proofs, Section 3 presents a non-constructive proof of statement \eqref{AC} as a theorem when every $A$ is well-ordered, in Section 4 we explore generalizations to partially ordered sets with a least element and to unordered sets over which relations with a unique element exist, and Section 5 provides the necessary and sufficient conditions under which the existence of a choice function can be proved within ZF without AC.

\section{Axiomatic Basis}

In this paper we base our work on the Zermelo-Fraenkel set theory (ZF), so all objects we  deal with are sets. For the sake of clarity we will tend to use capital letters such as $A, B, X, Y$ when emphasizing that an object is primarily considered as a collection of other objects. Conversely, lowercase letters like $a, b, x, y$ will refer to elements within sets. The capital letter $P$ will be used to denote sets of ordered pairs, while $R$ and $Q$ will specifically denote relations. The lowercase letter $f$ will be reserved for functions.

This paper is based entirely on the following axioms of ZF, all of which were present in \cite{Zermelo1908}:

1. Axiom of extensionality
\begin{equation}\label{AExt}
\forall{A} \forall{B} \forall{x} \big( (x\in A \iff x\in B) \implies A=B \big)
\end{equation}

2. Axiom of pairing
\begin{equation}\label{APair}
\forall{x} \forall{y} \exists{A} \forall{z} \big( z \in A \iff (z=x \lor z=y) \big)
\end{equation}
\textit{Note:} When $x=y$, this axiom allows the creation of singletons, i.e. sets containing exactly one element: $A = \{x\}$.

3. Axiom of union
\begin{equation}\label{AUni}
\forall{A_S} \exists{A_U} \forall{a} \big( a \in A_U \iff \exists{A}(A \in A_S \land a \in A) \big)
\end{equation}

4. Axiom of power set
\begin{equation}\label{APow}
\forall{A} \exists{\mathcal{P}(A)} \forall{B} \big( B \in \mathcal{P}(A) \iff \forall{x}(x \in B \implies x \in A) \big)
\end{equation}
\textit{Note:} $B \subseteq A \iff \forall{x}(x \in B \implies x \in A)$.

5. Axiom schema of separation
\begin{equation}\label{ASep}
\forall{w_1,w_2,...,w_n} \forall{A} \exists{B} \forall{x} \big( x \in B \iff (x \in A \land \varphi(x,w_1,w_2,...,w_n,A)) \big)
\end{equation}
Note: $\varphi(x,w_1,w_2,...,w_n,A)$ is any well-formulated formula (wff), in which $B$ is not free.

Given these axioms, we can define: 
\begin{itemize}
    \setlength{\itemsep}{0pt}
    \setlength{\parskip}{0pt}
    \setlength{\parsep}{0pt}
    \item an ordered pair $(x,y)$,
    \item the Cartesian product of two sets $A \times B$ as the set of all ordered pairs $(x,y)$ where $x\in A,y\in B$,
    \item a function $f:A \rightarrow B$ as a subset $f \subseteq A \times B$ such that $\forall{x\in A}\exists!{y\in B}(f(x)=y \iff (x,y)\in f)$,
    \item a relation $R$ over $A$ as a subset $R \subseteq A \times A$.
\end{itemize}

When dealing with order relations, we will use the notation $xRy$ interchangeably with $(x,y) \in R$, and we will call a relation $R$ over $A$ a well-order if 
\begin{equation}\label{WellOrd}
\begin{split}
\forall{x,y,z\in A}\big[ \big(xRy \lor yRx) & \land \big((xRy \land yRx) \implies x=y\big) \land \\ \big((xRy \land yRz) \implies xRz\big) \big] & \land \forall{B\subseteq A}\big(B\ne \emptyset \implies \exists{m\in B}\forall{b\in B}(mRb) \big)
\end{split}
\end{equation}

In this paper we can limit ourselves with much less restrictive orders, for example we will call a relation $R$ over $A$ a partial order with a least element if
\begin{equation}\label{PartOrdLeast}
\begin{split}
\forall{x,y,z\in A}\big[ \big(xRx) & \land \big((xRy \land yRx) \implies x=y\big) \land \\ \big((xRy \land yRz) & \implies xRz\big) \big] \land \big(A\ne \emptyset \implies \exists{m \in A}\forall{b\in A}(mRb) \big)
\end{split}
\end{equation}

Note that because $R$ in \eqref{WellOrd} and \eqref{PartOrdLeast} is anti-symmetric, the least element $m$ is unique for each set or subset (i.e. we can write $\exists!m$ instead of $\exists m$).

\section{The Theorem of Choice for Families of Well-Ordered Sets}

\begin{theorem}[Theorem of Choice, Well-Ordered Case]\label{TCwo}
 For every non-empty family of non-empty well-ordered sets there exists a function that maps each set of that family into an element of that set:
\begin{equation}\label{TC}
\begin{split}
 \forall{A_S} \Big[&\big(A_S \ne \emptyset \land \forall{A \in A_S} ( A \ne \emptyset \land \mathcal{R}_A \ne \emptyset\big) \implies \\& \exists {f_c \in \mathcal{P}(A_S \times \bigcup_{A \in A_S}{A} ) } \big[ \forall{A \in A_S} \exists! {a \in A} \big((A,a) \in f_c\big) \big]  \Big]
\end{split}
\end{equation}
where every ${R_A \in \mathcal{R}_A}$ is a well-order relation over $A$ as defined by \eqref{WellOrd}.

\begin{note}
\eqref{TC} is essentially the same statement as \eqref{AC} with the additional condition that there should exist at least one well-order $R_A$ for each $A$. Since there are no other changes, \eqref{TC} is applicable to any $A_S$, including uncountable families.   
\end{note}

\begin{proof}
Step 1. Combining the fact that each $A \in A_S$ admits a well-order with the fact every well-order has a unique least element, which follows from (7), 
\begin{equation}\label{a_min}
\forall{A \in A_S} \exists{R_A} \exists!{a^*_{R_A} \in A} \forall{b \in A} \big(a^*_{R_A}Rb\big)
\end{equation}

Step 2. Applying \eqref{AUni} to $A_S$, we can build $A_U = \bigcup_{A\in A_S}{A}$.

Step 3. Using \eqref{APair}, we can build a singleton $\{A\}$ for every $A \in A_S$, and with \eqref{APow} and \eqref{AUni} we can build the Cartesian product $P_A = \{A\} \times A$. Note that $P_A$ is unique for each $A$ and contains ordered pairs $(A,a)$ for every $a\in A$.

Step 3.1. Using \eqref{APow} and \eqref{AUni}, we can build two unions:
$$U_1=\bigcup_{A\in A_S}{\mathcal{P}(A \times A)}, U_2=\bigcup_{A\in A_S}{\mathcal{P}(P_A \times P_A)}$$

Step 3.2. For every well-order $R_A$ over $A\in A_S$ we can build a corresponding well-order $R_{P_A}$ over $P_A$ as follows:
$$\forall{A \in A_S}\exists R_{P_A} \forall{x} \big[ x\in R_{P_A} \iff (x\in P_A\times P_A) \land \exists{R_A}\in\mathcal{R}_A\big( x_{1,2} R_A x_{2,2} \big) \big]$$
where indices are used to denote elements of the ordered pair $x=(x_1,x_2)=\big((x_{1,1},x_{1,2}),(x_{2,1},x_{2,2})\big)$.

Step 3.3. It is evident that there exists a bijection between the orders $R_A$ and $R_{P_A}$, defined as follows (using set-builder notation for clarity):
\begin{equation}\label{p_bijection} 
\begin{split} 
p(R_A) = \big\{ \big((A,a),(A,b)\big) \in P_A\times P_A | a R_A b\big\}, \\
p^{-1}(R_{P_A}) = \big\{ (a,b) \in A\times A | (A,a) R_{P_A} (A,b)\big\}   
\end{split}
\end{equation}

Step 4. Using \eqref{ASep}, we can separate from $U_2$ its subset $Q_S$ containing only those sets of pairs $\big((A,a),(A,b)\big)$ for each $A$, for which there exists a well-order $R_A$ such that $a R_A b$:
\begin{equation}\label{QS_def}
\exists{Q_S}\forall{Q} \Big[ Q\in Q_S \iff \big(Q\in U_2 \land \varphi_1(Q) \big)\Big]
\end{equation}
where
\begin{equation}\label{phi1_def}
\varphi_1(Q) = \forall{A\in A_S} \exists{R_A} \forall{a,b\in A} \big[ \big((A,a),(A,b)\big)\in Q \iff a R_A b \big]
\end{equation}

Step 4.1. Define $Q_A = Q \cap (P_A \times P_A)$, which is effectively a restriction of $Q$ to those pairs related to the set $A$:
$$ Q_A = \big\{ \big((X,x),(Y,y)\big) \in Q | X=A \land Y=A \big\} $$
Thus, we've created a partition: $Q = \bigcup_{A\in A_S}{Q_A}$ and $A\ne B \implies Q_A \cap Q_B=\emptyset$.

Step 4.2 Notice that in \eqref{phi1_def}, the expression $\big((A,a),(A,b)\big)\in Q$ is equivalent to $\big((A,a),(A,b)\big)\in Q_A$. Therefore, we can rewrite \eqref{phi1_def} as
$$ \varphi_1(Q) = \forall{A\in A_S} \exists{R_A} \forall{a,b\in A} \big[ \big((A,a),(A,b)\big)\in Q_A \iff a R_A b \big]. $$
Using the bijection \eqref{p_bijection}, this is equivalent to
$$ \varphi_1(Q) = \forall{A\in A_S} \exists{R_A} \forall{a,b\in A} \big[ \big((A,a),(A,b)\big)\in Q_A \iff \big((A,a),(A,b)\big)\in p(R_A) \big], $$
and, applying \eqref{AExt}, we can further rewrite it as
\begin{equation}\label{phi1_def_alt}
\varphi_1(Q) = \forall{A\in A_S} \exists{R_A} \big( Q_A=p(R_A) \big)
\end{equation}
This indicates that $p^{-1}(Q_A)$ equals to some $R_A$, i.e. is a well-order over $A$, and every element $Q \in Q_S$ is a union of well-orders that is formed by combining exactly one well-order $R_{P_A}$ for each $A \in A_S$.

Step 5. Let's now prove that the following expression is true:
\begin{equation}\label{QA_has_least_element}
\forall{Q\in Q_S} \forall{A\in A_S} \exists{m \in A} \forall{b\in A} \big[ \big((A,m),(A,b)\big)\in Q \big]
\end{equation}

The simplest way involves noting that $\big((A,m),(A,b)\big)\in Q$ is equivalent to $\big((A,m),(A,b)\big)\in Q_A$, and since $Q_A$ is a well-order, it has a least element $(A,m)$. Therefore, the least element exists, as required.

Alternatively, assume that for some $Q \in Q_S$
$$ \exists{A\in A_S} \forall{m \in A} \lnot \forall{b\in A} \big[ \big((A,m),(A,b)\big)\in Q \big], $$
then we can replace $Q$ its subset $Q_A$:
$$ \exists{A\in A_S} \forall{m \in A} \lnot \forall{b\in A} \big[ \big((A,m),(A,b)\big)\in Q_A \big], $$
and since this is true for every $m$, it also holds for the least element of any well-order over $A$, including that particular $R_A$ that equals to $p^{-1}(Q_A)$:
$$ \exists{A\in A_S} \forall{R_A \in A} \forall{a^{*}_{R_A}} \lnot \forall{b\in A} \big[ \big((A,a^{*}_{R_A}),(A,b)\big)\in p(R_A) \big], $$
$$ \exists{A\in A_S} \forall{R_A \in A} \forall{a^{*}_{R_A}} \lnot \forall{b\in A}  \big(a^{*}_{R_A} R_A b\big), $$
contradicting the existence of a least element $a^{*}_{R_A}$ as established in \eqref{a_min}.

Step 5.1. Using a similar reasoning, we can also prove that $Q_S$ is non-empty. Assume that $Q_S = \emptyset$, i.e. $\forall{Q\in U_2}\big[Q \notin Q_S\big]$, then, using \eqref{QS_def} and \eqref{phi1_def_alt},

$$ \forall{Q\in U_2}\big[ \lnot \forall{A\in A_S} \exists{R_A} \big( Q_A=p(R_A) \big) \big], $$
\begin{equation}\label{contradict_QS_nonempty}
\forall{Q\in U_2}\big[ \exists{A\in A_S} \forall{R_A} \big( p^{-1}(Q_A) \ne R_A \big) \big],   
\end{equation}
but since $p^{-1}$ is a bijection, and all possible sets of pairs on $P_A \times P_A$ are covered (since the formula holds for any $Q \in U_2$, and that includes every $Q_A$ for each $A$), then $p^{-1}(Q_A)$ spans the entire power set $\mathcal{P}(A\times A)$, and none of its elements are equal to some well-order $R_A$. Thus there are no well-orders over $A$. Therefore, our initial assumption made must be false, and $Q_S \ne \emptyset$.

Alternatively, we can note that \eqref{contradict_QS_nonempty} is equivalent to $$\forall{R'\in U_1} \big[ \exists{A\in A_S} \forall {R_A \in \mathcal{R}_A} \big( R'_A \ne R_A \big) \big],$$ or $\forall{R'\in U_1}\big[ \exists{A\in A_S} \big( R'_A \notin \mathcal{R}_A) \big]$, where $R'_A = R' \cap (A \times A)$ is a restriction of $R'$ to $A$. Since $U_1$ by construction contains all possible orderings for every $A \in A_S$, but none of its elements belongs to $\mathcal{R}_A$ for some $A$, there are no orderings in $\mathcal{R}_A$, i.e. $\exists{A \in A_S}(\mathcal{R}_A = \emptyset)$, which contradicts to the conditions of the theorem, so our initial assumption must be false, and $Q_S \ne \emptyset$.

Step 6. Using \eqref{ASep}, we can separate from $\mathcal{P}(A_S\times A_U)$ its subset $F_c$ containing only those sets of pairs $(A,m)$ for each $A$, in which $m$ satisfies \eqref{QA_has_least_element} for some $Q\in Q_S$:
\begin{equation}\label{Fc_def}
\exists{F_c}\forall{f_c} \Big[ f_c\in F_c \iff \big(f_c\in \mathcal{P}(A_S\times A_U) \land \varphi_c(f_c) \big)\Big]
\end{equation}
where, for $Q\in Q_S, A\in A_S$,
\begin{equation}\label{phic_def}
\varphi_c(f_c) = \exists{Q}\forall{A} \forall{m\in A} \Big[ (A,m)\in f_c \iff \forall{b\in A}\big[\big((A,m),(A,b)\big)\in Q\big]\Big]
\end{equation}
Note that on step 5.1 we demonstrated that $m$ satisfies \eqref{QA_has_least_element} if and only if $m = a^{*}_{R_A}$ for $R_A=p^{-1}\big(Q\cap(P_A\times P_A)\big)$, therefore \eqref{phic_def} can be rewritten as follows:
$$\varphi_c(f_c) = \exists{Q\in Q_S}\forall{A\in A_S}\forall{m\in A}\Big[ (A,m)\in f_c \iff m=a^{*}_{p^{-1}(Q_A)}\Big],$$
\begin{equation}\label{phic_def_alt}
\varphi_c(f_c) = \exists{Q\in Q_S}\forall{A\in A_S} \Big[ f_c(A)=a^{*}_{p^{-1}(Q_A)}\Big].
\end{equation}

Step 6.1. From \eqref{phic_def_alt} and \eqref{QA_has_least_element} it immediately follows that for every $Q\in Q_S$ and every $A\in A_S$ such an element $m=a^{*}_{p^{-1}(Q_A)}$ exists, so each $A$ is mapped to at least one element $m$. And, since a well-order has only one minimal element, this mapping is unique. Therefore, every $f_c \in F_c$ is a function.

Step 6.2. Let's prove that $F_c \ne \emptyset$. We can use the same logic as on step 5.1, first assuming that $F_c = \emptyset$, then $\forall{f_c\in \mathcal{P}(A_S\times A_U)}\lnot\varphi_c(f_c)$, and, using \eqref{phic_def_alt},
$$\forall{f_c\in \mathcal{P}(A_S\times A_U)}\forall{Q\in Q_S}\exists{A\in A_S} \Big[ f_c(A) \ne a^{*}_{p^{-1}(Q_A)}\Big].$$
Since this must hold for every pair $(A,a)\in A_S\times A_U$, it implies there is no such $m\in A$ that $(A,m)\in A_S\times A_U \land m=a^{*}_{p^{-1}(Q_A)}$:
$$\forall{Q\in Q_S}\exists{A\in A_S}\forall{m\in A} \Big[ (A,m)\notin A_S\times A_U \lor m\ne a^{*}_{p^{-1}(Q_A)}\Big],$$
$$\forall{Q\in Q_S}\exists{A\in A_S}\forall{m\in A} \Big[ m\ne a^{*}_{p^{-1}(Q_A)}\Big],$$
and, since $Q_S \ne \emptyset$, the statement above implies that 
$$\exists{Q\in Q_S}\exists{A\in A_S}\forall{m\in A} \Big[ m\ne a^{*}_{p^{-1}(Q_A)}\Big],$$
which contradicts to \eqref{QA_has_least_element}. Hence, our initial assumption $F_c=\emptyset$ must be false.

Finally, since $F_c\ne \emptyset$, the statement \eqref{TC} holds:
$$\exists{f_c}\in\mathcal{P}(A_S\times A_U)\forall{A \in A_S} \exists! {a \in A} \big((A,a) \in f_c\big)$$

\end{proof}

\end{theorem}

\section{The Theorem of Choice under Relaxed Assumptions and Its Implications}

In the previous section, we limited our consideration to families of well-ordered sets. However, in the proof of \eqref{TC} we only relied on two properties from \eqref{WellOrd}: the existence of a least element and anti-symmetry. Both properties were used in a very relaxed form. For example, anti-symmetry was only needed to ensure the least element is unique (i.e., using $\exists!$ in place of $\exists$), which was crucial for building the $f_c$ function, because it should map each set $A$ into one and only one of its elements. We did not require anti-symmetry for any other purpose. Similarly, we only needed the fact that each set $A$ contained a least element, but we never required any subset of $A$ to have a least elements as well. Our proof would still hold if a set itself has a least element, while some of its subsets do not. Therefore, we can relax the assumptions of \textbf{Theorem \ref{TCwo}} and formulate its generalized variants.

One way to generalize \textbf{Theorem \ref{TCwo}} is to define $\mathcal{R}_A$ as a set of partial orders with a least element, as specified by \eqref{PartOrdLeast}. In this case we do not require every subset of $A$ to have a least element, only the set itself. Anti-symmetry ensures this element in unique. Consequently, we can formulate the following generalization, and the same proof as in the previous section applies:

\begin{theorem}[Theorem of Choice, Partially Ordered Case]\label{TCpo}
 For every non-empty family of non-empty partially ordered sets with a least element there exists a function that maps each set of that family into an element of that set:
\[
\begin{split}
 \forall{A_S} \Big[&\big(A_S \ne \emptyset \land \forall{A \in A_S} ( A \ne \emptyset \land \mathcal{R}_A \ne \emptyset\big) \implies \\& \exists {f_c \in \mathcal{P}(A_S \times \bigcup_{A \in A_S}{A} ) } \big[ \forall{A \in A_S} \exists! {a \in A} \big((A,a) \in f_c\big) \big]  \Big]
\end{split}
\]
where every ${R_A \in \mathcal{R}_A}$ is a partial order with a least element over $A$ as defined by \eqref{PartOrdLeast}.
\end{theorem}

This generalization significantly broadens the scope of sets for which we can prove choice based on ZF axioms alone, without assuming AC. For example, we can prove the existence of a choice function for any family of the non-empty intervals of real numbers, and then build a total order on them (although, of course, without AC we are unable to well-order such sets). In fact, we can do this even without \textbf{Theorem \ref{TCpo}}, using a purely ``constructive'' approach, by proposing the following choice function for any family of non-empty intervals over $\mathbb{R}$:
$$
    f_c(A) = a^{*} =
    \begin{cases}
        0, & \text{if } A = \mathbb{R}, \\
        \dfrac{a + b}{2}, & \text{if } A = [a,b] \lor A=(a,b) \lor A=(a,b] \lor A = [a,b), \\
        a - 1, & \text{if } A = (-\infty,a) \lor A = (-\infty,a], \\
        a + 1, & \text{if } A = (a,\infty) \lor A = [a,\infty).
    \end{cases}
$$

\textit{Note: It is essential to emphasize that $f_c(A)$ is undefined in all cases not explicitly specified above.}

Whenever the ``constructive" approach works, there should be an equivalent way of constructing a function of choice with a single axiom of separation. In this particular case, to formulate an axiom that effectively separates from $\mathcal{P}(\mathbb{R}) \times \mathbb{R}$ only those sets of pairs $(A,x)$ where $A \subseteq \mathbb{R}$ is an interval, $x \in A$, and there is one and only one pair $(A, x)$ for each $A$, we can propose the following well formulated formula $\varphi_2(A,x)$:
\[
\begin{split}
\varphi_2(A,x) &= (A \subseteq \mathbb{R}) \land \big( (A=\mathbb{R}) \land x=0\big) \lor \exists{a,b\in\mathbb{R}}\forall{y\in\mathbb{R}} \big[ 
\\ &\big( ((y \in A \iff y \le a) \lor (y \in A \iff y < a)) \land x = a-1 \big) \lor  
\\ &\big( ((y \in A \iff y \ge a) \lor (y \in A \iff y > a)) \land x = a+1 \big) \lor  
\\ &\big( ((y \in A \iff a \le y \le b) \lor (y \in A \iff a < y < b) ) \land x = \dfrac{a+b}{2} \big) \lor
\\ &\big( ((y \in A \iff a < y \le b) \lor (y \in A \iff a \le y < b)) \land x = \dfrac{a+b}{2} \big)
\big] \big)   
\end{split}
\]
It is evident that $\varphi_2(A,x)$ holds only if $A$ is a non-empty interval and $x$ belongs to this interval.

To apply \textbf{Theorem \ref{TCpo}}, we need to show that at least one suitable partial order exists. It is straightforward to construct such an order using the $a_*$ defined above as the least point:
$$ \forall{x,y \in A}\big( xRy \iff |x-a^*| \le |y - a^*| \land (|x-a^*| \ne |y - a^*| \lor x \le y \big).$$
This partial order can be generalized to hyper-intervals in $\mathbb{R}^n$ by using a least vector $a^*$, with each component defined according to the same formula.

This example shows that using the ``constructive'' approach is much simpler, as proving the existence of a choice function in this manner is equivalent to formulating a single axiom of separation. In contrast, our proof of \textbf{Theorem \ref{TCpo}} is based on two consecutive applications of the schema of separation, each resulting in its own axiom. However, our approach is much more flexible, as it remains applicable even when we cannot explicitly construct a specific choice function or a specific partial order with a least element---such as when we know from some properties that a suitable partial order exists over a given set, but we are unable to write down a well-formulated formula for its least element.

Another possible generalization would be to abolish any requirements for $R_A$ to be a transitive, reflexive and anti-symmetric relation. Essentially, we only need $R_A$ to have one property: 
\begin{equation}\label{genOrder}
\exists!{m\in A}\forall{b\in A}\big((m,b)\in R_A \big),    
\end{equation}
which implies $(m,m)\in R_A$ and $\forall{x\in A}\big((x,m)\in R_A\big) \implies x=m$. Consequently, we can generalize \eqref{TC} to the case when $\mathcal{R}_A$ is a non-empty set of arbitrary relations (not necessary orders) satisfying \eqref{genOrder}. However, we can prove that \eqref{genOrder} is equivalent to the existence of some partial order with a least element using the following theorem:

\begin{theorem}\label{Tord}
A partial order with a least element over a non-empty set $A$ exists if and only if there can be written a well-formulated formula that is true for one and only one element of $A$.
\begin{proof}
Suppose there exists $R$, a partial order with a least element over $A$, as defined by \eqref{PartOrdLeast}, then $\exists!{m \in A}\forall{b\in A}(m R b)$. We can write a well-formulated formula $\varphi_R(A,a) = \forall{b \in A}(a R b)$. It's clear that $\varphi_R(A,a) \iff a=m$.

Conversely, suppose we can write a well-formulated formula $\varphi(A,a)$ such that $\exists!{m \in A} \forall{a \in A}\big(\varphi(A,a) \iff a=m)$. Then, we can construct a set $R$ with an axiom of separation: $\forall{A}\exists{R}\forall{x}(x\in R \iff (x\in A \times A) \land (x_1 = x_2 \lor \varphi(A,x_1)) )$, where $x=(x_1, x_2)$. Let's prove that $R$ is a partial order over A:

\begin{itemize}
    \setlength{\itemsep}{0pt}
    \setlength{\parskip}{0pt}
    \setlength{\parsep}{0pt}
    \item $aRa \iff (a = a \lor \varphi(A,a))$, which is always true ,
    \item $aRb \land bRa \iff (a=b \lor \varphi(A,a))\land(b=a \lor \varphi(A,b)) \iff (a=b) \lor (\varphi(A,a) \land \varphi(A,b)) \implies a=b$,
    \item suppose $aRb \land bRc$: if $a=b$, then $aRc$; if $b=c$, then $aRc$ as well; otherwise, $a \ne b \land b \ne c$, but since $aRb \land bRc \iff (a=b \lor \varphi(A,a))\land(b=c \lor \varphi(A,b))$, here it reduces to $aRb \land bRc \iff \varphi(A,a) \land \varphi(A,b)$, which is impossible when $a \ne b$; thus, in all feasible cases $aRb \land bRc \implies aRc$.
\end{itemize}

Finally, $\exists{m\in A}(\varphi(A,m)) \implies \exists{m\in A}\forall{b\in A}(mRb))$.
\end{proof}
\end{theorem}

\textbf{Corollary 1.} \textit{A partial order with a least element over set $A$ exists if and only if there exists a function that maps $A$ to one of its elements.}

\textbf{Corollary 2.} \textit{A partial order with a least element exists over any topological space that has a single point as its deformation retract, which includes any path-connected or contractible space.}

By combining \textbf{Theorem \ref{TCpo}} with these results, we can claim that the existence of a function of choice could be proved from ZF axioms alone, without any kind of AC, even in cases when a ``constructible" approach fails---i.e., when we cannot explicitly specify a canonical or definable element for a particular set in a manner similar to what we used for the intervals of $\mathbb{R}$ discussed earlier. 

For example, consider a family $\mathcal{F}_X$ of arbitrary contractible topological spaces. We cannot take any point as a canonical element to construct an apparent choice function. However, we know that there is a constant map $f_p:X \rightarrow \{x_0\}$, with $x_0\in X$, for every space $X\in \mathcal{F}_X$. Therefore, there exists a partial order with a least element $R_X$ such that $(x,y)\in R_X \iff (x,y)\in X\times X \land (x=y \lor x=f_p(X))$. Using \textbf{Theorem \ref{TCpo}}, we can then claim the existence of a choice function on any such family $\mathcal{F}_X$ without assuming AC. This reasoning can be extended from families of contractible spaces to families of path-connected spaces. In particular, we can prove, using ZF alone, the existence of a function of choice for any family of hyperballs and hyperspheres, which, unlike hyperintervals, was not achievable with a ``constructive" approach.

\section{Necessary and Sufficient Conditions for Proving Choice in ZF without AC}

In the previous sections we proved that for any family of partially ordered sets with a least element there exists a function of choice; this provides a sufficient condition for proving choice in ZF without AC. We also demonstrated that a partial order with a least element over a set exists if and only if there is a function that maps that set into one of its elements. Combining these results, we can now establish a necessary condition for proving choice in pure ZF: without assuming AC, choice can be proved only for the families of partially ordered sets with a least element.

\begin{theorem}\label{Tcond}
In ZF without AC, a function of choice $f_c$ as defined by \eqref{AC} exists for a non-empty family of sets $A_S$ if and only if for each $A \in A_S$, $A \ne \emptyset$ and there exists a partial order with a least element over $A$.
\begin{proof}
Suppose that $A_S \ne \emptyset \land \forall{A \in A_S}[A \ne \emptyset \land \mathcal{R}_A \ne \emptyset]$, where every $R_A \in \mathcal{R}_A$ is a partial order with a least element over A. By \textbf{Theorem \ref{TCpo}}, in this case, a function of choice $f_c$ exists for $A_S$ .

Conversely, suppose that a function of choice $f_c$ exists for $A_S$, i.e. $$\exists {f_c \in \mathcal{P}(A_S \times \bigcup_{A \in A_S}{A} ) } \big[ \forall{A \in A_S} \exists! {a \in A} \big((A,a) \in f_c\big) \big].$$ From this, it immediately follows that $\forall{A \in A_S}(A \ne \emptyset)$. While multiple such functions $f_c$ may exist, for any of them we can write a well-formulated formula $\varphi_{f_c}(A,x)=(x=f_c(A))$. It is clear that, given any $f_c$ (and we know that at least one such function exists), this formula is true for one and only one element of every $A \in A_S$. Then, applying \textbf{Theorem \ref{Tord}}, we can conclude that there exists a partial order with a least element over every $A \in A_S$, and the least element in that particular case is exactly $f_c(A)$.

Alternatively, if this approach seems insufficiently rigorous and appears to rely on picking a particular $f_c$, we can apply the schema of separation as follows:
$$ \forall{f_c}\forall{A}\exists{\mathcal{R}_A}\forall{R}\big( R \in \mathcal{R}_A \iff R \in \mathcal{P}(A \times A) \land \varphi_3(R,A,f_c) \big),$$
where
$$\varphi_3(R,A,f_c) = \forall{a,b \in A}\big((a,b)\in R \iff a=b\lor a=f_c(A)\big)$$

Using the same reasoning as in the proof of \textbf{Theorem \ref{Tord}}, we can show that every $R \in \mathcal{R}_A$ is a partial order with a least element over $A$. 

Thus, it only remains to prove that $\mathcal{R}_A \ne \emptyset$ or, equivalently, that $\varphi_3(R,A,f_c)$ is not false for every possible $R \in \mathcal{P}(A \times A)$. To do this, let's consider two binary relations over a given $A$. First, the identity function, $id_A = \big\{ (a,a) | a \in A \big\}$---it is clear that $id_A \subset A \times A$ and, since $A\ne\emptyset$, we have $id_A\ne\emptyset$. Second, a binary relation generated on a given set $A$ using a given function of choice $f_c$, $f^{-1}_{A,f_c} = \big\{ (a,b) | a=f_c(A) \land b \in A \big\}$---it is clear that $f^{-1}_{A,f_c} \subset A \times A$ and, since $\exists{f_c}\land A\ne\emptyset$, we have $f^{-1}_{A,f_c}\ne\emptyset$. Now, for any $A\in A_S$ and any function of choice $f_c$, there exists a union $(id_A \cup f^{-1}_{A,f_c}) \subset A \times A$, and it is quite evident that $\varphi_3(id_A \cup f_{A,f_c},A,f_c)$ is true. Therefore, $\mathcal{R}_A \ne \emptyset$.
\end{proof}
\end{theorem}

\section{Conclusion}

In this paper, we developed a non-constructive approach to proving the existence of a function of choice for specific families of sets. The key difference between our method and existing ``constructive" approaches is that we do not rely on specifying the choice function explicitly, as was done in \cite{Hrbacek1999} to show that choice can be proved for any well-ordered set. Their approach inevitably requires assuming a particular property for the given sets or canonical (definable) elements in order to write down a formula for the particular function of choice. Our approach takes an entirely different path: instead of explicitly specifying some function of choice, we use the axiom schema of separation to separate a subset of all possible choice functions from the set of all binary relations, and then prove that subset is non-empty. Consequently, our approach works even in situations where it is impossible to specify a canonical element in each set or identify a specific property to construct a function of choice.

Although our initial focus was on families of well-ordered sets, and we developed a non-constructive alternative to existing ``constructive" proofs, it quickly became clear that our approach can be applied in a more general context. We attempted to relax some of the assumptions implied by the well-ordering property. In fact, our method does not depend on any assumptions about the subsets of the given sets, nor does it rely on transitivity or totality. This naturally extends our results to families of partial ordered sets with a least element. Next, we tried to investigate whether we could prove choice given families of unordered sets, and found that our approach would require every set to have a definable element, but it can be proved that in this case a partial order with a least element can be constructed over such set. Based on this, we established a necessary and sufficient condition for proving choice in ZF without AC: a function of choice exists for a non-empty family if any only if there is a partial order with a least element over every set in that family.

Finally, some applications of our results were examined. We showed that choice can be proved in ZF for any family of non-empty intervals in $\mathbb{R}$: first, by explicitly constructing a particular function of choice (effectively, using a ``constructible" approach); second, by demonstrating how a partial order with a least element can be built, allowing us to apply our theorem to prove choice for this family. The same reasoning can be extended to families of non-empty hyper-intervals in $\mathbb{R}^n$. We also analyzed more tricky cases involving families of hyper-balls, hyper-spheres, and even arbitrary contractible and path-connected topological spaces. While the ``constructive" approach appears to fail in these cases, our method can prove the existence of choice function for such families in pure ZF, without assuming AC. And, as a further consequence, since any countable set is well-ordered, we can also claim that the countability of a union of countably many countable sets can be proved in ZF without AC.

\section*{Declarations}

\paragraph{Funding and Competing interests.} The author did not receive support from any organization for the submitted work. The author has no financial or proprietary interests in any material discussed in this article.

\paragraph{AI Usage.} DeepAI (deepai.org) with Standard model was used solely for grammar correction, language editing, and validation of formulas and proofs that had previously been developed by the author. All scientific content, interpretations, and conclusions were generated by the author. The final text was reviewed and verified by the author.

\bibliography{main}

\end{document}